\documentclass[11pt]{amsart}
\usepackage{amssymb}
\usepackage{hyperref}
\usepackage{color}

\newcommand{\R}{{\mathbb R}}

\numberwithin{equation}{section}
\newtheorem{theorem}{Theorem}[section]

\newtheorem{prop}[theorem]{Proposition}

\allowdisplaybreaks
\def\eqref#1{(\ref{#1})}

\theoremstyle{definition}

\theoremstyle{remark}

\newtheorem{remark}[theorem]{Remark}

\def\R{\mathbb R}
\def\C{\mathbb C}
\def\D{\mathbf D}

\begin{document}

\author
{V.~G. Maz'ya}
\address{Department of Mathematics, Link\"oping University, SE-581 83, Link\"oping,
Sweden and RUDN University, 
6 Miklukho-Maklay St, Moscow, 117198, Russia
}
\email{vlmaz@mai.liu.se}

\author{I.~E. Verbitsky}\address{Department of Mathematics, University of Missouri, Columbia, Missouri 65211, USA}
\email{verbitskyi@missouri.edu}
\title[Accretivity of the second order differential operator]{Accretivity of the general second order linear differential operator}


\begin{abstract}

For the  general 
 second order linear differential operator 
$$\mathcal L_0 = \sum_{j, \, k=1}^n \, a_{jk} \, \partial_j \partial_k +  
\sum_{j=1}^n \, b_{j} \,  \partial_j + c$$
with complex-valued 
distributional coefficients $a_{jk}$, $b_{j}$, and $c$  in an open set $\Omega \subseteq \R^n$ ($n \ge 1$), we present conditions    
which ensure that $-\mathcal L_0$ is accretive, i.e., 
${\rm Re} \, \langle -\mathcal L_0 \phi, 
\phi\rangle \ge 0$ for all $\phi \in C^\infty_0(\Omega).
$ 
\end{abstract}

\dedicatory{Dedicated to Carlos Kenig with admiration and deep respect}

\maketitle
\tableofcontents

\section{Introduction}\label{Introduction}

Let $\mathcal{L}_0$ be the general second order 
differential operator in an open set $\Omega \subseteq \R^n$,  
\begin{equation}\label{E:0.0}
\mathcal{L}_0= \sum_{j, \, k=1}^n \, a_{jk} \, \partial_j \partial_k +  
\sum_{j=1}^n \, b_{j} \,  \partial_j + c,  
\end{equation}
where $a_{jk}$,  $b_{j}$, and $c$ are 
complex-valued distributions in $D'(\Omega)$. In this paper, we are concerned with the \textit{accretivity} of  $-\mathcal{L}_0$ defined in terms  of the real part of its quadratic form: 
 \begin{equation}\label{accr'}
{\rm Re} \,  \langle -\mathcal{L}_0 \, u, \, u \rangle \ge 0,  
\end{equation}
for all complex-valued functions  $u \in C^\infty_0(\Omega)$.  In other words, we study the dissipativity property associated with  $\mathcal{L}_0$.

If  the principal part $\mathcal{A} u$ of the differential operator is given 
in the divergence form, 
\begin{equation}\label{E:div-form}
\mathcal{A} u =  {\rm div} \, (A \nabla u), \quad u \in C^\infty_0(\Omega),
\end{equation}
then we consider the operator 
\begin{equation}\label{E:0.b}
\mathcal{L} u=  {\rm div} \, (A \nabla u) +  \mathbf{b} \cdot\nabla u + c \, u, 
\end{equation}
with distributional coefficients $A=(a_{jk})$, $\mathbf{b} =(b_j)$, and $c$. 
The corresponding sesquilinear form $\langle \mathcal{L} u, v\rangle$ is given by 
\begin{equation}\label{E:div-sesq}
\langle \mathcal{L} u, v\rangle = - \langle A \nabla u, \nabla v \rangle 
+ \langle \mathbf{b} \cdot\nabla u, v\rangle + \langle c \,  u, v\rangle. 
\end{equation}

We observe that $\mathcal{L}_0 = \mathcal{L} - {\rm Div} \, {A} \cdot  \nabla$ 
(see, for instance, 
\cite{KP}, \cite{MV3}), 
where ${\rm Div}\!\!:  D'(\Omega)^{n\times n} \to D'(\Omega)^{n}$ is the row divergence operator defined in Sec. \ref{Section 2}. 
 Hence, we 
 can always express $\langle \mathcal{L}_0 u, v\rangle$ in the form 
\eqref{E:div-sesq}, with  $ \mathbf{b}- {\rm Div} \, {A}$ 
in place of $ \mathbf{b}$, for 
distributional coefficients $A$ and  $\mathbf{b}$. 

If the differential operator is given in a more general divergence form, 
\begin{equation}\label{E:0.b*}
\mathcal{L}_1 u=  {\rm div} \, (A \nabla u) +  \mathbf{b_1} \cdot\nabla u + {\rm div}\, 
(  \mathbf{b_2} \, u) + c_1 \, u, 
\end{equation}
then obviously  it is reduced to \eqref{E:0.b} with $\mathbf{b} =\mathbf{b_1} + 
\mathbf{b_1}$ and $c=c_1 + {\rm div}\,  \mathbf{b_2}$.

From now on, without loss of generality we will treat the accretivity property   
\begin{equation}\label{accr}
{\rm Re} \,  \langle -\mathcal{L} \, u, \, u \rangle \ge 0,  \quad \textrm{for all} \, \, u \in C^\infty_0(\Omega),
\end{equation}
 associated 
with  the divergence form operator \eqref{E:0.b}.

Assuming that 
$A=(a_{jk})$, 
$\mathbf{b}=(b_j)$ and $c $ are locally integrable in $\Omega$,   
 we write the sesquilinear form of $\mathcal{L}$ as 
  \begin{align}\label{sesq-l}
\langle \mathcal{L} u, v\rangle & =
 \int_{\Omega}  ( -(A \, \nabla u) \cdot \nabla \overline{v}  + \mathbf{b} \cdot\nabla u \, \, 
\overline{v} + c \, u \,  \overline{v} ) \, dx,      
\end{align}
where $u, \, v  \in C^\infty_0(\Omega)$.   Sometimes it will be convenient to write  
(\ref{E:div-sesq}) in this form even for distributional coefficients $A$, $\mathbf{b}$, and $c$.

Our main results on the accretivity problem are stated in Sec. \ref{main} below, in particular,  
Proposition \ref{prop_1} and Theorem \ref{thm-i2} in higher dimensions $n \ge 2$, 
and Theorem \ref{Theorem 0} in the one-dimensional case.

This problem is of substantial interest in the real-variable case as well, where 
the goal is to characterize  operators $-\mathcal{L}$ with real-valued coefficients whose quadratic form is nonnegative definite, 
\begin{equation}\label{pos-def}
\langle -\mathcal{L} h, h\rangle \ge 0, \quad \textrm{for all real-valued} \, \, h \in C^\infty_0(\Omega).
 \end{equation} 
 Such operators $-\mathcal{L}$ are called nonnegative definite. 

 In the special case of Schr\"{o}dinger operators  
 \begin{equation}\label{schro}
 \mathcal{H} u = {\rm div} \, (P \nabla u) + \sigma \, u,
  \end{equation}
 with real-valued  $P\in D'(\Omega)^{n \times n}$ and $\sigma \in D'(\Omega)$, a characterization of this property 
 was obtained earlier in \cite[Proposition 5.1]{JMV1} under the assumption that $P$ is uniformly elliptic, i.e., 
 \begin{equation}\label{ell}
m \, ||\xi||^2 \le  P(x) \xi \cdot \xi \le M \, ||\xi||^2, \quad \textrm{for all} \, \, \xi \in \R^n, 
\, \,  \textrm{a.e.} \, \, x \in \Omega, 
 \end{equation}
 with the ellipticity  constants $m>0$ and $M<\infty$. 
  
An analogous characterization of   \eqref{pos-def} 
 for more general operators which include drift terms, $\mathcal L 
= {\rm div}  ( P \nabla \cdot) +
\mathbf{b} \cdot \nabla + c$, with real-valued coefficients and $P$ satisfying \eqref{ell},  
is given in Theorem \ref{thm-real}  below.

Returning to the accretivity problem 
\eqref{accr} for $\mathcal L 
= {\rm div}  ( A \nabla \cdot) +
\mathbf{b} \cdot \nabla + c$
in the complex-valued case,  define the symmetric component $A^s$ and  its skew-symmetric counterpart $A^c$  
respectively by 
 \begin{equation}\label{trans}
A^s =\frac{1}{2}(A + A^{\perp}), \quad A^c =\frac{1}{2}(A - A^{\perp}),
 \end{equation}
where $A=(a_{jk})\in D'(\Omega)^{n \times n}$, and $A^{\perp}=(a_{kj})$ is the transposed matrix. 

As we will see below, in order that $\mathcal{L}$ be accretive, the matrix 
 $A^s$  must have a nonnegative definite real part: 
$P = {\rm Re} \, A^s$  should satisfy 
\begin{equation}\label{matr-def}
P \xi\cdot \xi \ge 0 \quad \textrm{for all} \, \, \xi \in R^n, \quad \textrm{in} \, \, D'(\Omega). 
\end{equation}
Moreover, the corresponding Schr\"{o}dinger operator $\mathcal{H}$ defined by 
\eqref{schro} 
with 
\[
P= {\rm Re} \, A^s, \quad \sigma ={\rm Re} \, c - \frac {1}{2}{\rm div} \, ({\rm Re} \, \mathbf{b}), 
\]
must be 
nonnegative definite: 
\begin{equation}\label{norm}
[h]_{\mathcal{H}}^2 =\langle - \mathcal{H} h, h\rangle= 
\langle P \nabla h, \nabla h\rangle - \langle \sigma h, h\rangle \ge 0, 
\end{equation}
for all real-valued (or complex-valued) $h \in C^\infty_0(\R^n)$.

The rest of the accretivity problem for $\mathcal L$ (see Sec. \ref{sec2.1}) 
boils down to the commutator 
inequality involving these quadratic forms,  
\begin{equation}\label{comm-i}
\left \vert \langle \tilde{\mathbf{b}},  u \nabla v -v \nabla u\rangle \right\vert 
\le  [u]_{\mathcal{H}} \, [v]_{\mathcal{H}},  
\end{equation}
for all real-valued  $u, v \in C^\infty_0(\R^n)$, where the real-valued vector field 
$\tilde{\mathbf{b}}$ is given by 
\[
\tilde{\mathbf{b}} =  \frac{1}{2} \,  \left [ {\rm Im} \,\mathbf{b} - {\rm Div} ({\rm Im} \, A^c) \right ].
\] 

Under some mild restrictions on $\mathcal{H}$, the ``norms''  $[u]_{\mathcal{H}}$  and $[v]_{\mathcal{H}}$  on the right-hand side of  \eqref{comm-i}
 can be replaced,  up to a constant multiple, with the corresponding Dirichlet norms $|| \nabla \cdot||_{L^{2} (\Omega)}$. This leads to explicit criteria of accretivity in terms of  
 ${\rm BMO}^{-1}$ estimates, 
 as  
 in Theorem \ref{thm-i2} below.

Similar commutator inequalities related to compensated compactness theory \cite{CLMS}  
 were studied earlier \cite{MV3}
 in the context of the 
\textit{form boundedness} problem   
 for  $\mathcal{L}$,  
\begin{equation}\label{E:1.2old}
| \langle \mathcal{L} \, u, \, v \rangle | \leq \, C \, || \nabla u||_{L^{2} (\Omega)} \, 
|| \nabla v||_{L^{2} (\Omega)},    \quad u, \, v  \in C^\infty_0(\Omega), 
\end{equation}
where the constant $C$ does not depend on $u, \, v $.

If \eqref{E:1.2old} holds, 
then 
$\langle \mathcal{L} \, u, \, v \rangle$ can be extended by continuity to 
$u, v \in L^{1, \, 2} (\Omega)$. Here $L^{1, \, 2}(\Omega)$   is 
the completion of (complex-valued) $C^\infty_0(\Omega)$ functions with respect to the  norm 
$|| u||_{L^{1, \, 2}(\Omega)} = ||\nabla  u||_{L^2(\Omega)}$. 
 Equivalently,  
\begin{equation}\label{E:dual-op}
\mathcal L : \, \, L^{1, \, 2}(\Omega)\to L^{-1, \, 2}(\Omega)
\end{equation}
 is a bounded operator, where  
$L^{-1, \, 2}(\Omega)=L^{1, \, 2}(\Omega)^*$ is a dual 
Sobolev space. Analogous problems have been studied  for the inhomogeneous Sobolev space 
 $W^{1, \, 2}(\Omega)= L^{1, \, 2}(\Omega)\cap L^2(\Omega)$, fractional Sobolev spaces, 
 infinitesimal 
 form boundedness, and other related questions (\cite{MV2}--\cite{MV5}).

The form boundedness problem \eqref{E:1.2old} for the 
general second order differential operator $\mathcal{L}$ in the case $\Omega =\R^n$ 
was  characterized completely by the authors in \cite{MV3} using harmonic 
analysis  and potential theory methods. We observe that no ellipticity assumptions 
were imposed in \cite{MV3} 
on the principal part $\mathcal{A}$ of $\mathcal{L}$. 

For the Schr\"{o}dinger operator 
$\mathcal{H}=\Delta + \sigma$ with $\sigma \in D'(\Omega)$, where 
either $\Omega =\R^n$, or $\Omega$ is a bounded domain that supports 
Hardy's inequality (see \cite{An}), a characterization  of form boundedness was obtained earlier in \cite{MV1}. 
A different approach for $\mathcal{H}=  {\rm div} \, (P \nabla \cdot) + \sigma$ in 
general open sets $\Omega\subseteq \R^n$ based on PDE 
and real analysis methods,  under the uniform ellipticity assumptions on $P$, was developed 
in \cite{JMV1}. A quasilinear version for operators of the $p$-Laplace type can be found in \cite{JMV2}.

Both the accretivity and form boundedness problems have numerous applications, including  mathematical quantum mechanics (\cite{RS}, \cite{RSS}),  elliptic and parabolic PDE with singular coefficients (\cite{EE}, \cite{GGM}, \cite{K}, 
\cite{KP}, \cite{M}, \cite{NU}, \cite{H}, \cite{Ph}), fluid mechanics and Navier-Stokes equations (\cite{GP}, \cite{KT}, \cite{SSSZ}, \cite{T}), semigroups  
and Markov processes  
(\cite{LPS}), homogenization theory (\cite{ZP}), harmonic analysis 
(\cite{CLMS}, \cite{FNV}), etc. 

We remark that, for the form boundedness, the assumption that the coefficients are complex-valued is not essential. It is easily reduced to the real-valued case. 

The situation 
is quite different for the accretivity problem, where the presence of complex-valued coefficients  
leads to additional complications, especially in higher dimensions ($n \ge 2$) when the matrix ${\rm Im} \, A$ is not symmetric, and/or the imaginary part of $\mathbf{b}$ is nontrivial. 
Then commutator inequalities of the type \eqref{comm-i} with sharp constants, ${\rm BMO}$ estimates, 
and other tools of harmonic analysis come into play.  

These phenomena, along with some examples demonstrating possible interaction between 
the principal part, drift term and zero-order term of the operator $\mathcal{L}$, are discussed in the next section.

\section{Main results}\label{main}

\subsection{General accretivity criterion}\label{sec2.1} 

Let $\Omega \subseteq \R^n$ ($n \ge 1$) be an open set,  and let $\mathcal{L}$ be a divergence 
form second order linear differential operator with complex-valued distributional 
coefficients defined by \eqref{E:0.b}. 

For $A=(a_{jk})\in D'(\Omega)^{n \times n}$, define its symmetric part $A^s$ and  skew-symmetric part $A^c$  respectively by \eqref{trans}. 
The accretivity property for $-\mathcal{L}$ can be characterized  in terms of the following real-valued expressions: 
\begin{equation}\label{expr}
P={\rm Re} \, A^s, \quad  \tilde{\mathbf{b}}=  \frac{1}{2} \,  \left [ {\rm Im} \, \mathbf{b} - {\rm Div} \,  ( {\rm Im} \, A^c) \right ], \quad \sigma={\rm Re} \, c - \frac {1}{2}{\rm div} \, ({\rm Re} \, \mathbf{b}),
\end{equation}
where $P=(p_{jk})\in D'(\Omega)^{n\times n}$, $\tilde{\mathbf{b}}=(\tilde{b_j})\in D'(\Omega)^{n}$, 
and $\sigma\in D'(\Omega)$. This is a consequence of the relation (see Sec. \ref{proof-1})
\begin{equation}\label{L-2}
{\rm Re}  \langle -\mathcal{L} u, u\rangle 
={\rm Re}  \langle -\mathcal{L}_2 u, u\rangle, \quad 
u\in C^\infty_0 (\Omega),
\end{equation}
where 
 \begin{equation}\label{L-2a}
  \mathcal{L}_2  = {\rm div} \, (P \nabla \cdot) 
  + 2 i \, \tilde{\mathbf{b}} \cdot \nabla + \sigma. 
 \end{equation}

Moreover, in order that $-\mathcal{L}$ be accretive,  the matrix $P$ must be nonnegative definite, i.e., $P \xi\cdot \xi \ge 0$ in 
$D'(\Omega)$ 
for all $\xi \in \R^n$. In particular, each  $p_{jj}$ ($j=1, \ldots, n$) is  a nonnegative Radon measure. 
  
  A comprehensive characterization 
  of accretive operators $-\mathcal{L}$ is given in the following proposition. 
  
\begin{prop}\label{prop_1}  Let $\mathcal L 
= {\rm div}  ( A \nabla \cdot) +
\mathbf{b} \cdot \nabla + c$, where  $A \in D'(\Omega)^{n\times n}$, 
 $\mathbf{b} \in D'(\Omega)^n$ and $c \in D'(\Omega)$ are complex-valued.  
 Suppose that $P$, $\tilde{\mathbf{b}}$, and $\sigma$ are defined by \eqref{expr}. 
 
 The operator $-\mathcal L$ is accretive if and only  if  $P$ is a nonnegative definite matrix, and 
the following two conditions hold: 
\begin{equation}\label{E:1.b}
[h]_{\mathcal{H}}^2=  \langle P \nabla h, \nabla h\rangle  - \langle \sigma \,  h, h\rangle\ge 0,  
\end{equation}
for all real-valued  $h \in C^\infty_0(\Omega)$,  and 
\begin{equation}\label{E:1.c}
 \left \vert \langle \tilde{\mathbf{b}},  u \nabla v -v \nabla u\rangle \right\vert 
\le [u]_{\mathcal{H}} \, [v]_{\mathcal{H}},  
\end{equation}
for all real-valued  $u, v \in C^\infty_0(\Omega)$.  
\end{prop}

In equation \eqref{E:1.b}, the expression   $[h]_{\mathcal{H}}^2=\langle - \mathcal{H} h, h\rangle$ stands for  the quadratic form associated with the 
Schr\"{o}dinger operator  
$ \mathcal{H}={\rm div} \, (P \nabla h) + \sigma$,  
discussed in Sec. \ref{sec2.3}. 

In Theorem \ref{thm-grad} below, we show that it is possible to replace 
$\tilde{\mathbf{b}}$ in Proposition \ref{prop_1} by $\tilde{\mathbf{b}}-P \nabla \lambda$, with an appropriate change 
in $\sigma$. In particular, the commutator condition \eqref{E:1.c} trivializes if $\tilde{\mathbf{b}}=P \nabla \lambda$. This reduction is used in Sec. \ref{one-dim} in the one-dimensional case.

\subsection{Real-valued coefficients}\label{sec2.2} 

As a consequence of Proposition \ref{prop_1}, we see that, for operators with real-valued coefficients, 
the sole condition \eqref{E:1.b} characterizes nonnegative definite 
operators $-\mathcal{L}$ in an open set $\Omega\subseteq \R^n$ ($n\ge 1$). 
   We next state a more explicit characterization of  this property,  
  under 
 the assumption that  $P=A^s\in L^1_{{\rm loc}} (\Omega)^{n \times n}$  in the sufficiency part, and 
 that $P$ is uniformly elliptic in the necessity part. 
 
\begin{theorem}\label{thm-real}
Let $\mathcal{L} = {\rm div}  ( A \nabla \cdot) +
\mathbf{b} \cdot \nabla + c$, where  $A \in D'(\Omega)^{n\times n}$, 
 $\mathbf{b} \in D'(\Omega)^n$ and $c \in D'(\Omega)$ are real-valued. 
 Suppose that $P=A^s\in L^1_{{\rm loc}} (\Omega)^{n \times n}$ is a nonnegative definite matrix a.e. 
 
 {\rm (i)} If  there exists a measurable vector field $\mathbf{g}$ in $\Omega$ such that  $(P \mathbf{g}) \cdot \mathbf{g} \in L^1_{{\rm loc}}(\Omega)$, and 
 \begin{equation}\label{g-cond}
   \sigma=c - \frac {1}{2}{\rm div} \, (\mathbf{b}) \le {\rm div} \, (P \mathbf{g}) 
   - (P \mathbf{g}) \cdot \mathbf{g} \quad \textrm{in} \, \, D'(\Omega), 
      \end{equation}
then the operator 
$-\mathcal{L}$ is nonnegative definite. 

 {\rm (ii)} Conversely, if $-\mathcal{L}$ is nonnegative definite,  then there exists a vector field $\mathbf{g} \in L^2_{{\rm loc}} (\Omega)^n$ so that 
 $(P \mathbf{g}) \cdot \mathbf{g} \in L^1_{{\rm loc}}(\Omega)$, and 
 \eqref{g-cond} holds, 
    provided $P$ is uniformly elliptic.  
\end{theorem}

The uniform ellipticity condition on $P$ in statement (ii) of Theorem \ref{thm-real} can be relaxed. We intend to address this question elsewhere. 

Conditions similar to \eqref{g-cond}  are well known
 in ordinary differential equations, in relation  to disconjugate Sturm-Liouville equations and Riccati equations with continuous coefficients  
 (\cite[Sec. XI.7]{Ha}, Corollary 6.1, Theorems 6.2 and 7.2). See also 
 \cite{H}, \cite{MV2}, as well as  
 the discussion in Sec. \ref{sec2.4} and 
Sec. \ref{one-dim} below in the one-dimensional case.

\subsection{Schr\"{o}dinger operators}\label{sec2.3} 

As was mentioned above, in the special case  of Schr\"{o}dinger operators 
$ \mathcal{H}={\rm div} \, (P \nabla h) + \sigma$,  with real-valued  $\sigma \in D'(\Omega)$ and uniformly elliptic 
$P$, Theorem \ref{thm-real} was obtained originally in 
\cite[Proposition 5.1]{JMV1}. Under these assumptions, $-\mathcal{H}$ is nonnegative definite, i.e., 
\[
[h]_{\mathcal{H}}^2=\langle -\mathcal{H} h, h \rangle\ge 0, \quad \textrm{for all} \, \, h \in C^\infty_0(\Omega), 
\]
 if and only if there exists a vector field $\mathbf{g} \in 
        L^2_{{\rm loc}} (\Omega)^n$ such that 
    \begin{equation}\label{g-cond1}
   \sigma \le {\rm div} \, (P \mathbf{g}) 
   - (P \mathbf{g}) \cdot \mathbf{g} \quad \textrm{in} \, \, D'(\Omega).
      \end{equation}

A ``linear'' sufficient condition for $-\mathcal{H}$ to be nonnegative definite is given by 
$
 \sigma \le {\rm div} \, (P \mathbf{g}), 
$
where $\mathbf{g} \in L^2_{{\rm loc}} (\Omega)^n$ satisfies the inequality
\[
\int_\Omega (P \mathbf{g}\cdot \mathbf{g}) \, h^2 \, dx \le \frac{1}{4}\int_\Omega 
(P \nabla h\cdot P \nabla h) \, dx, \quad \textrm{for all} \, \, h \in C^\infty_0(\Omega).
\]
 Here $P \mathbf{g}\cdot \mathbf{g}$ is an \textit{admissible} measure (Sec. \ref{Section 2}). However, 
such conditions are not necessary, with any constant in place of $ \frac{1}{4}$, even when $P=I$; see 
\cite{JMV1}.

We recall that in Proposition \ref{prop_1} above, 
 the  nonnegative definite quadratic form $[h]_{\mathcal{H}}^2$ is associated with  the Schr\"{o}dinger operator $\mathcal{H}$  with real-valued coefficients  
 $
  P ={\rm Re} \, A^s,
 \quad 
 \sigma={\rm Re} \, c - \frac {1}{2}{\rm div} \, ({\rm Re} \, \mathbf{b}). 
 $
 
 Hence,  \eqref{g-cond1} characterizes 
      the first condition of Proposition \ref{prop_1}  given by 
      \eqref{E:1.b}. The second one, the commutator condition \eqref{E:1.c}, will be discussed further in Sections \ref{sec2.5} 
      and \ref{sec2.6}; see also an example in Sec. \ref{ex}.

Notice that, even for form bounded $\sigma$ such that 
 \begin{equation}\label{f-bdd}
\vert \langle \sigma, h^2\rangle \vert \le C \, ||\nabla h||^2_{L^2(\Omega)}, 
\quad \textrm{for all} \, \, h \in C^\infty_0(\Omega), 
 \end{equation}
the fact that  $-\mathcal{H}$ is nonnegative definite
 is not equivalent 
to the existence of a positive solution  $u$ to the Schr\"{o}dinger equation $\mathcal{H} u=0$. In other words, in our setup, the Allegretto-Piepenbrink theorem is generally not true. 
See \cite{JMV1}, \cite{MV1}, \cite{MV2},  and the literature cited there for further discussion.

\subsection{The one-dimensional case}\label{sec2.4} 

In the one-dimensional case, it is possible to avoid problems 
with  commutator estimates 
using methods  of ordinary differential equations (\cite{Ha}, \cite{Hi}). In particular, the following  theorem gives a generalization of Theorem \ref{thm-real} for complex-valued coefficients in the one-dimensional case. In the statements below we will make use of the standard convention $\frac{0}{0}=0$.  

\begin{theorem}\label{Theorem 0} Let $I \subseteq \R$ be an open interval (possibly unbounded). Let $a, b, c  \in D'(I)$, and $\mathcal{L} u = (a \, u')' + b u' + c$. Suppose that $p={\rm Re} \, a  \in L^1_{{\rm loc}} (I)$, and ${\rm Im} \,  b \in L^1_{{\rm loc}} (I)$. 

{\rm (i) } The operator $-\mathcal{L}$ is accretive  
 if and only if  $\frac{({\rm Im} \,  b)^2}{p}\in L^1_{{\rm loc}}(I)$, where $p \ge 0$ a.e.,
and the following quadratic form inequality holds:
\begin{equation}\label{0.1}
\int_{I} p (h')^2 dx -\langle {\rm Re} \, c - \frac{1}{2}({\rm Re} \, b)',  h^2 
\rangle - \int_{I} \frac{({\rm Im} \, b)^2}{4p}\, h^2 \, dx \ge 0,
\end{equation}
for all real-valued $h \in C^{\infty}_0(I)$.

{\rm (ii) }    If  there exists a function 
$f \in L^1_{{\rm loc}}(I)$ such that $ \frac{f^2}{p} \in L^1_{{\rm loc}}(I)$, and 
\begin{equation}\label{0.2}
{\rm Re} \, c - \frac{1}{2}({\rm Re} \, b)' 
- \frac{({\rm Im} \, b)^2}{4 p}\le f' - \frac{f^2}{p} \quad \textrm{in} \, \, D'(I), 
\end{equation}
then the operator $-\mathcal{L}$ is accretive. 

Conversely, if $-\mathcal{L}$ is accretive, and 
$m\le p(x)\le M$ a.e. for some constants $M, m>0$, then 
there exists a function 
$f \in L^2_{{\rm loc}}(I)$ such that \eqref{0.2} holds. 
\end{theorem}

\begin{remark}
  Clearly, the function $f$ in \eqref{0.2} and $g$ in 
 Theorem \ref{thm-real} are related through $f = p \, g$. The condition that $p$ is uniformly bounded above and below by positive 
 constants in statement (ii)  of Theorem \ref{Theorem 0} can be relaxed to $\frac{1}{p} \in L^1_{{\rm loc}}(I)$, with $\frac{f^2}{p} \in L^1_{{\rm loc}}(I)$ in place of $f \in L^2_{{\rm loc}}(I)$. 
\end{remark}

\begin{remark}  As was mentioned above, the assumptions on $p={\rm Re} \, a \ge 0$ in Theorem \ref{Theorem 0} can be substantially relaxed. 
In general, $p$ is a Radon measure in $I$. It is easy to see 
that condition \eqref{0.1} with $\rho$ in place of $p$, where $\rho = \frac{d p}{dx}$ is the absolutely continuous part 
 of the measure $p$, is sufficient for  $-\mathcal{L}$ to be accretive. 
 
 On the other hand,  $-\mathcal{L}$ 
 is accretive if, for instance,  
$a =2 \delta_{x_0}$, 
 $ c=- 2  \delta_{x_0}$, and $b= i \delta_{x_0}$, where $x_0 \in I$. This example is immediate from 
 Proposition \ref{prop_1}. Operators with  measure-valued $A$ in the principal 
 part  $\mathcal{A} = {\rm div} \, (A\nabla\cdot)$ are treated in \cite{CiM} in the context of $L^p$-dissipativity.  
\end{remark}

The characterization of accretivity   obtained in Theorem 
\ref{Theorem 0} in the one-dimensional case does not involve 
${\rm Im} \, a$ and ${\rm Im} \, c$. However,  
${\rm Im} \, b$ plays an important role.  In higher dimensions, the situation is more complicated.  
 The term ${\rm Im} \, b$ may contain both the irrotational and divergence-free 
 components, the latter in combination 
with  
${\rm Im} \, A^c$. (See Theorem \ref{thm-i2}, and Example in Sec. \ref{ex} below.) 

There is an analogue of Theorem \ref{Theorem 0} in higher dimensions for operators 
with complex-valued coefficients, but only 
in the case where $\tilde{\mathbf{b}}$ has a specific form, for instance, if 
$\tilde{\mathbf{b}} = P \nabla \lambda$ for some $\lambda \in D'(\Omega)$. 
More general vector fields are treated in Theorem \ref{thm-grad} below.

\subsection{Upper  and lower bounds of quadratic forms}\label{sec2.5}

Returning now to general operators with complex-valued coefficients in the case $n \ge 2$, we recall that the first condition of Proposition \ref{prop_1}  is  necessary 
 for the accretivity of $-\mathcal{L}$, namely, 
\begin{equation}\label{nec}
  \langle  \sigma \,  h, h\rangle 
\le  \int_\Omega (P \nabla h \cdot \nabla h) \, dx,  
\end{equation}
for all real-valued $h \in C^\infty_0(\Omega)$, where $\sigma = {\rm Re } \, c - \frac{1}{2}{\rm div} ({\rm Re} \,  \mathbf{b}) \in D'(\Omega)$, and 
${\rm Re} \, A^{s}=P \in D'(\Omega)^{n \times n} $ is a nonnegative definite matrix. 

Suppose now that $\sigma$
has a slightly smaller upper form bound, that is, 
\begin{equation}\label{E:ia}
  \langle  \sigma \,  h, h\rangle 
\le (1-\epsilon^2)  \int_\Omega (P \nabla h\cdot \nabla h) \, dx, \quad h \in C^\infty_0(\Omega),  
\end{equation}
for some $\epsilon \in (0, 1]$. We also consider the corresponding lower bound, 
\begin{equation}\label{E:ib}
  \langle  \sigma \,  h, h\rangle 
\ge - K  \int_\Omega (P \nabla h\cdot \nabla h) \, dx, \quad h \in C^\infty_0(\Omega),  
\end{equation}
for some  constant $K\ge 0$. 

Such restrictions on  real-valued $\sigma \in D'(\Omega)$ were 
invoked in \cite{JMV1}, for uniformly elliptic $P$.  

\begin{remark} Notice that 
\eqref{E:ia} is obviously satisfied for any $\epsilon \in (0, 1)$, 
up to an extra term  $C \, ||h||^2_{L^2(\Omega)}$, 
 if 
$\sigma$ is \textit{infinitesimally form bounded}  (\cite{RS}), i.e.,
\[
\left \vert 
\langle \sigma, h^2\rangle
\right \vert\le \epsilon \, || \nabla h||^2_{L^2(\Omega)} + 
C(\epsilon) \, ||h||^2_{L^2(\Omega)}, \quad h \in C^\infty_0(\Omega),   
\]
for any $\epsilon\in (0, 1)$. This property was characterized in 
 \cite{MV5}. The second term  on the right is sometimes  included in the definition of accretivity of the operator $-\mathcal{L}$. We can always incorporate it 
as a constant term in $\sigma -C(\epsilon)$. The same is true with regards to the lower bound where we can use $\sigma +C(\epsilon)$.
\end{remark}

If both bounds \eqref{E:ia} and \eqref{E:ib} hold for some $\epsilon\in (0, 1]$ and $K\ge 0$, then obviously 
\[
\epsilon   \int_\Omega (P \nabla h\cdot \nabla h) \, dx \le [h]^2_{\mathcal{H}} \le (K+1)^{\frac{1}{2}}  \int_\Omega ( P \nabla h \cdot \nabla h) \, dx, \quad h \in C^\infty_0(\Omega). 
\]

Assuming that $P$ satisfies the uniform ellipticity assumptions \eqref{ell}, we see that in this case 
condition \eqref{E:1.c} is equivalent, up to a constant multiple, to 
\begin{equation}\label{E:i.e}
\left \vert \langle \tilde{\mathbf{b}},  u \nabla v -v \nabla  u\rangle \right\vert 
\le C \,  ||\nabla u||_{L^2(\Omega)} \, ||\nabla v||_{L^2(\Omega)}
\end{equation}
where $C>0$ is a constant which does not depend on real-valued $u, v \in C^\infty_0(\Omega)$. 

This commutator inequality was characterized completely 
in the case $\Omega=\R^n$ in \cite[Lemma 4.8]{MV3} 
for complex-valued $u, v$. Clearly, that characterization works also 
in the case of real-valued $u, v$ as well (with a change of the constant $C$ up to a factor of $\sqrt{2}$).

\subsection{Main theorem on the entire space}\label{sec2.6} 

Combining the characterization of the commutator inequality \eqref{E:i.e} 
obtained in  \cite{MV3} with 
Proposition \ref{prop_1}, we deduce our main theorem in the case $\Omega=\R^n$. We employ separately the lower  bound \eqref{E:ib}  in the necessity part, and the upper bound \eqref{E:ia}  in the sufficiency part.

\begin{theorem}\label{thm-i2} Let $\mathcal L$ be a second order differential 
operator in divergence form \eqref{E:0.b} with complex-valued coefficients $A \in D'(\R^n)^{n\times n}$, 
 $\mathbf{b} \in D'(\R^n)^n$ and $c \in D'(\R^n)$ ($n \ge 2$).  Let $P$, 
 $\tilde{\mathbf{b}}$ and $\sigma$ be given by \eqref{expr}, where $P$ is uniformly 
 elliptic. 
 
  {\rm (i) } Suppose that $-\mathcal{L}$ is accretive, i.e.,  \eqref{accr}  holds, and 
suppose that  \eqref{E:ib} holds for some $K\ge 0$. 
 
 {\rm (a) } If $n \ge 3$,  then  $\tilde{\mathbf{b}}$  can be represented in the form 
\begin{equation}\label{E:ig}
\tilde{\mathbf{b}}  = \nabla f + {\rm Div} \, G, 
\end{equation} 
where $f\in D'(\R^n)$ is real-valued, and there exists a 
positive constant $C$ so that 
\begin{equation}\label{E:iga}
\int_{\R^n} |\nabla f |^2 h^2  dx \le C \, \int_{\R^n} |\nabla h|^2 dx, \quad \textrm{for all} \, \, h \in C^\infty_0(\R^n), 
\end{equation} 
 and $G \in {\rm BMO}(\R^n)^{n\times n}$ is 
a real-valued skew-symmetric matrix field. 

Moreover,   
$f$ and  $G$ above can be defined explicitly as 
\begin{equation}\label{E:igb}
f = \Delta^{-1} ( {\rm div } \, \tilde{\mathbf{b}}), \quad 
G =\Delta^{-1} ( {\rm Curl} \, \tilde{\mathbf{b}}),
\end{equation}
where the constant $C$ in \eqref{E:iga} and the ${\rm BMO}$-norm of $G$ may depend on $K$. 

 {\rm (b) } If $n=2$, then 
$\tilde{\mathbf{b}} = (-\partial_2 g, \partial_1 g)$, where 
 $g \in {\rm BMO}(\R^2)$ is  a real-valued function so that  ${\rm div} (\tilde{\mathbf{b}})=0$.    
 \smallskip

{\rm (ii) } Conversely, suppose that  \eqref{E:ia} holds 
for some $\epsilon \in (0, 1]$. Then $-\mathcal{L}$ is accretive  
if representation \eqref{E:ig} holds when $n \ge 3$, or $f=0$ and $\tilde{\mathbf{b}} = (-\partial_2 g, \partial_1 g)$ when $n=2$, so that both the constant $C$ in \eqref{E:iga} and the 
${\rm BMO}$-norm of $G$ (or $g$ when $n=2$) 
are small enough,  
depending only on  $\epsilon$. 
\end{theorem}

\begin{remark}  Notice that the condition  imposed   on 
the divergence-free component  ${\rm Div} \, G$ in the Hodge decomposition \eqref{E:ig}
 is much weaker that the condition   
on the irrotational component $\nabla f$.

 In particular, 
 \eqref{E:iga}   means that $|\nabla f|^2 dx \in 
\mathfrak{M}^{1,2}(\R^n)$ is an \textit{admissible measure}. Several equivalent characterizations of  the class $\mathfrak{M}^{1,2}(\R^n)$ are discussed in Sec. \ref{Section 2}  below. 
\end{remark}

\begin{remark} Under the assumptions of Theorem \ref{thm-i2}, 
$
\tilde{\mathbf{b}}\in {\rm BMO}^{-1}(\R^n)^n$ (see Sec. \ref{Section 2}). A thorough discussion of the  
space ${\rm BMO}^{-1}(\R^n)$ and its applications is given in \cite{KT}. 
\end{remark}  

\begin{remark} In \eqref{E:igb},  the Newtonian potential $\Delta^{-1}$ is understood  in terms of the weak-$*$ ${\rm BMO}$ convergence (see \cite{MV3}, \cite{St}), and 
${\rm Div}$ and ${\rm Curl}$ are the usual matrix operators defined in Sec. \ref{Section 2}. 

In the case $n=3$, we can use the usual vector-valued ${\rm curl} (\mathbf{g}) \in D'(\R^3)^3$ in place of ${\rm Div} \, G$ in decomposition \eqref{E:ig}, with   
$\mathbf{g} =\Delta^{-1} ( {\rm curl} \, \tilde{\mathbf{b}})$ in \eqref{E:igb}.
\end{remark}

\subsection{Example}\label{ex} 

We conclude Sec. \ref{main} with an example in two dimensions that demonstrates possible interaction between the principal part and lower order terms in the accretivity 
problem for operators with complex-valued coefficients. 

 Consider the operator $\mathcal{L}= {\rm div} \, (A \nabla \cdot)+\mathbf{b} \cdot \nabla + c$ in $\R^2$ 
with $A=(a_{jk})$, where $a_{11}=a_{22}=1$, $a_{12}=-a_{21} =i \, \lambda \, \log |x|$,  $\mathbf{b}=-x |x|^2$ and $c=-2|x|^2$, 
where $x \in \R^2$ and $\lambda \in \R$. 

If $|\lambda| \le  C$, where $C$ is an absolute constant, then by statement (ii) of  Theorem \ref{thm-i2}, the operator 
 $-\mathcal{L}$ is accretive due to the interaction between the principal part, the drift term,  and 
 the zero-order term  (harmonic oscillator).

In this example $P=I$, $\sigma= 0$, $\tilde{\mathbf{b}}= (-\partial_2 g, \partial_1 g)$, where $g = \frac{\lambda}{2}  \log |x|\in {\rm BMO} \, (\R^2)$. The upper bound 
\eqref{E:ia} obviously holds for any $\epsilon\in (0, 1]$, but the lower bound  \eqref{E:ia} 
 fails. 
 
 We note in passing that, by Proposition \ref{prop_1}, the optimal value of the constant 
 $|\lambda|$ 
in this example  is found  from the inequality
 \begin{equation}\label{jacob}
 \left \vert \int_{\R^2} g(x) \, J[u,v] \, dx
 \right \vert \le  ||\nabla u||_{L^2(\R^2)} ||\nabla u||_{L^2(\R^2)}, 
 \end{equation}
 for all real-valued $u, v \in C^\infty_0(\R^2)$, where $J[u,v] = \frac{\partial u}{\partial x_1} \frac{\partial v}{\partial x_2} - \frac{\partial u}{\partial x_2}  \frac{\partial v}{\partial x_1}$ is the determinant 
 of the Jacobian matrix $\mathbf{D}(u, v)$ (see \cite{CLMS}, \cite{MV3}, and 
 Sec. \ref{Section 2} below).

\section{Proofs of Proposition \ref{prop_1} and Theorem \ref{thm-real}}\label{proof-1}

\begin{proof}[Proof of Proposition \ref{prop_1}.]
Clearly, the principal part of ${\rm Re} \, \langle \mathcal{L} u, u\rangle$ 
depends only on ${\rm Re} \, A^s$ and ${\rm Im} \, A^c$, since for $u =f + i g\in D'(\Omega)$ ($f, g$ are real-valued), we have 
\begin{equation*}\label{E:1.1b}
 \begin{split}
  -{\rm Re} \,  \langle \mathcal{A} \, u, u \rangle  &  =  \langle {\rm Re} A^s \, \nabla f, \, \nabla f \rangle + 
   \langle {\rm Re} A^s \, \nabla g, \, \nabla g \rangle - 2 \, \langle {\rm Im} A^c \, \nabla f, \, \nabla g \rangle\\ & =  
    \langle {\rm Re} A^s \, \nabla f, \, \nabla f \rangle + 
   \langle {\rm Re} A^s \, \nabla g, \, \nabla g \rangle + 2 \, \langle 
   {\rm div} \, ({\rm Im} A^c \, \nabla f), \,  g \rangle 
   \\ & =  
    \langle {\rm Re} A^s \, \nabla f, \, \nabla f \rangle + 
   \langle {\rm Re} A^s \, \nabla g, \, \nabla g \rangle   \\ & \, \, +  \langle 
   {\rm div} \, ({\rm Im} A^c \, \nabla f), \,  g \rangle -  \langle 
   {\rm div} \, ({\rm Im} A^c \, \nabla g), \,  f \rangle.
    \end{split}
\end{equation*}

Since $A^c$ is skew-symmetric, it follows that  
\[
{\rm div} \, (A^c \nabla u) = - {\rm Div}(A^c) \cdot \nabla u \quad {\rm in} \, \, D'(\Omega),
\] 
where the vector field 
${\rm Div}(A^c)$ is solenoidal (divergence free). 
In particular, 
\[
{\rm div} \, ({\rm Im} A^c \, \nabla f) = -{\rm Div}({\rm Im} \, A^c) \nabla f,  
\]
for real-valued $f$. 

Letting $\mathbf{b_1}=  \mathbf{b} -  \, {\rm Div}(A^c) $, we see 
that the skew symmetric-part $A^c$ can always be included in the first-order term $\mathbf{b}_1 \cdot \nabla$, and hence 
does not affect the principal part of $\mathcal{L}$. Consequently, 
we have 
\begin{equation*}
\begin{split}
\langle -\mathcal{L} u, u\rangle & = \langle P \nabla f, \nabla f\rangle 
+ \langle P \nabla g, \nabla g\rangle 
- \langle  \mathbf{b_1}, f \nabla f + g \nabla g\rangle \\ &
- i \langle  \mathbf{b_1}, f \nabla g - g \nabla f\rangle -\langle  c, f^2 + g^2\rangle, 
\end{split}
\end{equation*}
and 
\begin{equation*}
\begin{split}
{\rm Re} \, \langle -\mathcal{L} u, u\rangle & = 
\langle P \nabla f, \nabla f\rangle 
+ \langle P \nabla g, \nabla g\rangle  - \langle {\rm Re} \,  \mathbf{b_1}, f \nabla f + g \nabla g\rangle 
\\ & + \langle {\rm Im} \,  \mathbf{b_1}, f \nabla g - g \nabla f\rangle -\langle {\rm Re} \,  c, f^2 + g^2\rangle.
\end{split}
\end{equation*}

Integrating by parts, 
and using the fact that ${\rm div} \, ({\rm Re} \,  \mathbf{b_1})={\rm div} \, ({\rm Re} \,  \mathbf{b})$, 
we deduce  
\[
\langle {\rm Re} \,  \mathbf{b_1}, f \nabla f  + g \nabla g\rangle  = - \frac{1}{2} 
\langle {\rm div} \, ({\rm Re} \,  \mathbf{b}), f^2  + g^2\rangle.
\]
It follows  that 
\begin{equation*}
\begin{split}
{\rm Re} \, \langle -\mathcal{L} u, u\rangle & = 
\langle P \nabla f, \nabla f\rangle 
+ \langle P \nabla g, \nabla g\rangle  \\ & + \langle {\rm Im} \,  \mathbf{b_1}, f \nabla g - g \nabla f\rangle - 
\langle  \sigma, f^2 + g^2\rangle,
\end{split}
\end{equation*}
where 
$
\sigma={\rm Re} \,  c - \frac{1}{2} {\rm div} \, ({\rm Re} \,  \mathbf{b}). 
$

This proves that ${\rm Re} \, \langle \mathcal{L} u, u\rangle = {\rm Re} \, \langle \mathcal{L}_2 u, u\rangle$, where $\mathcal{L}_2$ is defined by \eqref{L-2a}. 
Thus, \eqref{L-2} holds. 

Interchanging the roles of $f$ and $g$ we deduce that ${\rm Re} \, \langle \mathcal{L} u, u\rangle\ge 0$ if and only if 
\begin{equation*}
 \langle P \nabla f, \nabla f\rangle 
+ \langle P \nabla g, \nabla g\rangle    - \langle   \sigma, f^2 + g^2\rangle  \ge  \left \vert  \langle {\rm Im} \,  \mathbf{b_1}, f \nabla g - g \nabla f\rangle \right \vert. 
\end{equation*}

Using the quadratic form  $[f]^2_{\mathcal{H}}$ defined by \eqref{E:1.b}, 
we rearrange the preceding inequality  as follows,
\begin{equation}\label{rearr}
 \left \vert  \langle {\rm Im} \,  \mathbf{b_1}, f \nabla g - g \nabla f\rangle \right \vert  \le [f]^2_{\mathcal{H}} + [g]^2_{\mathcal{H}}.
\end{equation}
for all   are real-valued $f, g \in C^\infty_0(\Omega)$. Clearly,  the right-hand side of this inequality equals 
$
[f]^2_{\mathcal{H}} + [g]^2_{\mathcal{H}}= [u]^2_{\mathcal{H}},
$
where $[u]^2_{\mathcal{H}}=\langle - \mathcal{H} u, u\rangle\ge 0$ for every complex-valued  $u\in C^{\infty}_0(\Omega)$.  In particular, $-\mathcal{H}$ is nonnegative definite.

 Replacing $f, g$ in \eqref{rearr} with $\alpha f, \frac{1}{\alpha} g$ respectively, and minimizing over all real $\alpha\not=0$, we  deduce 
that ${\rm Re} \, \langle \mathcal{L} u, u\rangle\ge 0$ if and only if 
\begin{equation}\label{mult}
  \left \vert  \langle {\rm Im} \,  \mathbf{b_1}, f \nabla g - g \nabla f\rangle \right \vert \le 2 \, [f]_{\mathcal{H}}  [g]_{\mathcal{H}},
\end{equation}
where $f, g \in C^{\infty}_0(\Omega)$ are real-valued, provided $[u]^2_{\mathcal{H}}\ge 0$ for every complex-valued (or equivalently real-valued) $u\in C^{\infty}_0(\Omega)$.  Clearly, $\frac{1}{2} {\rm Im} \,   \mathbf{b_1}=  \tilde{\mathbf{b}}$, where 
$ \tilde{\mathbf{b}}$ is defined by \eqref{expr}, so that \eqref{mult} coincides 
with \eqref{E:1.c}.  

It remains  to show that if $-\mathcal{L}$ is an accretive operator, then 
$P={\rm Re} \, A^s$ is a non-negative definite matrix. Let 
$u= e^{it \, x \cdot \xi} \, v$, where $v\in C^\infty_0(\Omega)$ is real-valued, 
$t \in \R$ and $\xi \in \R^n$. Then clearly, 
\begin{equation*}
\begin{split}
\langle -\mathcal{L}u, u\rangle = & t^2 \langle (A \xi \cdot \xi) \, v, v\rangle +
 \langle A  \nabla v, \nabla v\rangle 
+ i t \langle (A \xi ) v, \nabla v\rangle\\ &
 - i t \langle A \nabla v, v \xi \rangle- i t \langle (\mathbf{b} \cdot \xi)v, v\rangle 
- \langle  \mathbf{b} \cdot \nabla v, v \rangle - \langle  c \,  v, v \rangle. 
\end{split}
\end{equation*}

It follows, 
\begin{equation*}
\begin{split}
 {\rm Re} \, \langle -\mathcal{L}u, u\rangle  = & t^2 \langle (P \xi \cdot \xi) \, v, v\rangle + \langle P  \nabla v, \nabla v\rangle 
- t \langle ({\rm Im} \, A) \xi  v, \nabla v\rangle \\ &
+  t \langle ({\rm Im} \, A)  \nabla v, v \xi \rangle + t   \langle ({\rm Im} \mathbf{b} \cdot \xi)v, v\rangle  - \langle  \sigma \,  v, v \rangle\ge 0.
\end{split}
\end{equation*}

Dividing both sides by $t^2$ and letting $t\to \infty$, we immediately get that 
 $\langle (P \xi \cdot \xi) \, v, v\rangle \ge 0$ for every real-valued $v\in C^\infty_0(\Omega)$. Then, for any $h \in C^\infty_0(\Omega)$, $h \ge 0$, denote 
 by $\eta\in C^\infty_0(\Omega)$ a cut-off function such that $\eta \, h =h$. 
 Setting $v= \eta (h + \delta)^{\frac{1}{2}} \in C^\infty_0(\Omega)$, for $\delta>0$, we see that 
 $\langle P \xi \cdot \xi, h+ \delta \, \eta^2\rangle \ge 0$. Letting $\delta\to 0$ yields 
  $P \xi \cdot \xi\ge 0$ in $D'(\Omega)$. This completes the proof 
of Proposition \ref{prop_1}. 
\end{proof}

\begin{proof}[Proof of Theorem \ref{thm-real}.]
We recall some estimates for non-negative definite, symmetric matrices $P=(p_{jk})$,  
starting with the 
Schwarz inequality 
\begin{equation}\label{swarz}
\vert P  \xi \cdot \eta \vert \le  \Big(P  \xi \cdot \xi\Big)^{\frac{1}{2}}  \Big(P \eta \cdot \eta\Big)^{\frac{1}{2}}, \quad \textrm{for all} \,\, \xi, \eta \in \R^n.
\end{equation}
From \eqref{swarz} with $\eta=P \xi$, we deduce the estimate  
\[
|P \xi|^2 \le  ||P|| \, (P \xi\cdot \xi), \quad \textrm{for all} \,\, \xi\in \R^n,
\]
where $||P||$ is the operator norm of $P$. Since  $||P||\le \sum_{j, k=1}^n \, |p_{jk}|$, 
using the preceding inequality with $\xi=\mathbf{g}$, we deduce, for any $h \in C^\infty_0(\Omega)$,
\[
\int_\Omega |P \mathbf{g}| \, h^2 \,  dx \le \left( \int_\Omega \Big(P \mathbf{g}\cdot \mathbf{g}
\Big) \, 
 h^2 \,  dx \right)^{\frac{1}{2}} \left( \int_\Omega (\sum_{j, k=1}^n \, | p_{jk}(x) |) \, 
 h^2 \,  dx \right)^{\frac{1}{2}}.
\]

We now prove statement (i) of Theorem \ref{thm-real}. 
From the preceding estimate it follows that $P \in L^1_{{\rm loc}} (\Omega)^{n\times n}$ and $(P \mathbf{g}) \cdot \mathbf{g} 
        \in L^1_{{\rm loc}} (\Omega)^n$ yield 
$P \mathbf{g} \in L^1_{{\rm loc}} (\Omega)^n$.

 Applying \eqref{swarz} with $\xi=\mathbf{g}(\cdot)$ and $\eta=\nabla h(\cdot)$, we obtain
\begin{equation*}
 \begin{split}
 \left \vert  \int_{\Omega}   [(P \mathbf{g}) \cdot \nabla h] \, h \, dx \right \vert 
& \le  \int_{\Omega}  [(P \mathbf{g}) \cdot \mathbf{g}]^{\frac{1}{2}}  
[(P \nabla h) \cdot \nabla h]^{\frac{1}{2}}  \, h \, dx 
\\ & \le  \left(\int_{\Omega}  [(P \mathbf{g}) \cdot \mathbf{g}]  \, h^2 \, dx 
\right)^{\frac{1}{2}}  
\left(\int_{\Omega}  [(P \nabla h) \cdot \nabla h]  \, dx \right)^{\frac{1}{2}}.  
\end{split}
\end{equation*}

Using \eqref{g-cond},  along with the preceding inequality, and integrating by parts, we estimate, 
\begin{equation*}
 \begin{split}
& \langle \sigma, h^2 \rangle \le \langle {\rm div} \, (P \mathbf{g}), h^2 \rangle
   - \int_{\Omega}  [(P \mathbf{g}) \cdot \mathbf{g}] \,  h^2 dx 
      \\
   &  = -2   \int_{\Omega}   [(P \mathbf{g}) \cdot \nabla h] \, h \, dx - 
   \int_{\Omega}  [(P \mathbf{g}) \cdot \mathbf{g}] \,  h^2 dx 
    \\ &\le 2 \, \left ( \int_{\Omega}  [(P \mathbf{g}) \cdot  \mathbf{g}] \,  h^2 dx\right)^{\frac{1}{2}}
   \left ( \int_{\Omega}  [(P \nabla h) \cdot \nabla h] \, dx\right)^{\frac{1}{2}}
    - \int_{\Omega}  [(P \mathbf{g}) \cdot \mathbf{g}] \,  h^2 dx 
   \\ & \le \int_{\Omega}  [(P \nabla h) \cdot \nabla h] \, dx.
\end{split}
\end{equation*}

In other words, \eqref{E:1.b} holds.  Since $\tilde {\mathbf{b}} =0$, and hence the commutator condition \eqref{E:1.c}   
is vacuous, $-\mathcal{L}$ is nonnegative definite 
by Proposition \ref{prop_1}. 
This proves statement (i) of Theorem \ref{thm-real}. 

Statement (ii) of Theorem \ref{thm-real} is immediate from Proposition \ref{prop_1}, and  the corresponding result 
for the Schr\"{o}dinger operator $\mathcal{H}$ in Sec. \ref{sec2.3}, which yields the existence 
of $\mathbf{g} \in L^2_{{\rm loc}} (\Omega)$ such that \eqref{E:1.b} holds.  The proof of  Theorem \ref{Theorem 0} 
is complete.
\end{proof}

\section{Proof of Theorem \ref{Theorem 0}}\label{one-dim}

 Let $\mathcal{L} u = (a \, u')' + b u' +c$ be a second order linear 
differential operator with complex-valued distributional coefficients $a$, $b$, $c$ on an open interval $I\subseteq \R$ (possibly unbounded). As in \eqref{expr}, we define 
the associated real-valued distributions $p$, $\tilde b$, and $\sigma$ by 
\[
p= {\rm Re} \, a, \quad \tilde{b} =\frac{1}{2} \, {\rm Im} \, b, \quad \sigma = {\rm Re} \, c-\frac{1}{2} \, 
 {\rm Re} \, b'.
 \]

Proposition \ref{prop_1} gives 
a criterion of accretivity for $-\mathcal{L} $ in terms of the quadratic form inequality 
for $-\mathcal{H}$, 
\[
[h]^2_{\mathcal{H}}=\langle p h', h'\rangle - 
\langle \sigma \, h, h\rangle \ge 0, \quad \textrm{for all} \, \, h\in C^\infty_0(I),
\]
where   $\mathcal{H} h = (p h')' + \sigma \, h$  is 
the Sturm-Liouville operator on $I$, together with the commutator inequality
\[
\left \vert \langle \tilde{b}, u v'-v u'\rangle \right \vert \le \, [u]^2_{\mathcal{H}} \, 
[v]^2_{\mathcal{H}},
\]
for all real-valued $u, v \in C^\infty_0(I)$.

In the case where $p, \, \tilde{b}  \in L^1_{{\rm loc}} (I)$, 
  it is possible 
to avoid the commutator inequality by including $\tilde{b}$ in the  
stronger  quadratic form inequality \eqref{0.1}, i.e., 
\begin{equation}\label{2.1}
[h]^2_{\mathcal{N}}= \int_I p \,   (h')^2 \, dt -\langle \sigma,  h^2 
\rangle -  \int_{I}  \frac{\tilde{b}^2}{p} \, h^2 \, dt \ge 0,
\end{equation}
for all real-valued $h \in C^{\infty}_0(I)$. Here $\mathcal{N} h = (p h')' +q \, h$,  with  
\begin{equation*}
q = {\rm Re} \, c - \frac{1}{2} ({\rm Re} \, b)' -  \frac{\tilde{b}^2}{p}. 
\end{equation*} 
This means that  $-\mathcal{L}$ is accretive if and only if  the  Sturm-Liouville operator $-\mathcal{N}$ is nonnegative definite. In this case, the condition $\frac{\tilde{b}^2}{p} \in L^1_{{\rm loc}} (I)$ is necessary for accretivity. 
We recall that throughout the paper, we  are using the convention 
$\frac{0}{0}=0$. 

The reduction of the accretivity of the operator $-\mathcal{L}$ 
to the property that the Schr\"{o}dinger operator $-\mathcal{N}$ 
is nonnegative definite is performed via an  exponential 
substitution $u = z e^{- i \lambda}$, where $\lambda \in D'(\Omega)$ is 
a real-valued vector field. It actually works in the general case  $n \ge 1$, provided $\mathbf{b}$ 
has a specific form, so that 
$\tilde{\mathbf{b}} = P \nabla \lambda$, under certain restrictions 
on $P$ and $\lambda$  (see Sec. \ref{grad} below). 
If $n=1$, we will show that this is always possible 
with  $\lambda'=\frac{\tilde{b}}{p} $ in $I$. 

We now prove 
statement (i) of Theorem \ref{Theorem 0}. Our first step is to show 
that, in the general case $\Omega\subseteq \R^n$, for $n \ge 1$, we can replace  
the operator $\mathcal{L}$ in the 
corresponding quadratic form inequalities 
with the operators $\mathcal{L_\epsilon}$  with  
mollified coefficients $A_\epsilon$, $\mathbf{b}_\epsilon$, and $c_\epsilon$ ($\epsilon>0$). 

More precisely, if $u\in C^\infty_0(\Omega)$, then we can pick a smooth real-valued 
cut-off 
function $\eta \in  C^\infty_0(\Omega)$ so that  $\eta \, u=u$. 
Then, applying the inequality ${\rm Re} \, \langle -\mathcal{L} u, u\rangle \ge 0$ 
with $\eta \, u =u$, we can clearly replace $A$, $\mathbf{b}$, and $c$ by 
$\eta^2 A$, $\eta^2  \mathbf{b}$, and $\eta^2 c$, respectively, so that we may assume that the coefficients 
of $\mathcal{L}$ are compactly supported.

Notice that  we can also 
replace $u(x)$ with  a shifted test function  $u_y(x)= u(x+y)$, where $|y|< \epsilon$, 
for a small enough $\epsilon$ depending on the support of $u$ and $\eta$. 
Let 
$\phi_\epsilon(y) = \epsilon^{-n} \phi(\frac{t}{\epsilon})$ be a mollifier,  so that $\phi \ge 0$, 
$\phi\in C^\infty(B(0, 1))$
and 
$\int_{B(0, 1)} \phi(y) \, dy=1$.  Then 
$\textrm{supp} \, (\phi_\epsilon) \subset B(0, \epsilon)$. 
Integrating both sides of the inequality  ${\rm Re} \, \langle -\mathcal{L} u_y, u_y\rangle \ge 0$ against $\phi_\epsilon(t) \, dt$, we obtain 
\[
{\rm Re} \, \langle -\mathcal{L_\epsilon} u, u\rangle \ge 0,
\]
where the coefficients 
$A_\epsilon$, $\mathbf{b}_\epsilon$, and $c_\epsilon$ are the mollifications 
of the distributions $\eta^2 A$, $\eta^2  \mathbf{b}$, and $\eta^2 c$, respectively. Conversely, 
if this inequality holds for small enough $\epsilon>0$, then passing to the limit 
as $\epsilon \to 0$, we recover the inequality 
${\rm Re} \, \langle -\mathcal{L} u, u\rangle \ge 0$. In other words, we can assume 
without loss of generality that the coefficients of $\mathcal{L}$ in the inequality 
${\rm Re} \, \langle -\mathcal{L} u, u\rangle \ge 0$ 
are $C^\infty_0(\Omega)$ functions. 

Returning to the one-dimensional case, we fix $u \in C^\infty_0(I)$, and define 
$p_\epsilon$, $\tilde{b}_\epsilon$, $\sigma_\epsilon$ as the mollifications of $p$, $\tilde{b}$, 
and $\sigma$, respectively. Obviously, we can always replace $p$ in the inequality 
${\rm Re} \, \langle -\mathcal{L} u, u\rangle \ge 0$  by $p + \delta$, for some 
$\delta>0$, and 
eventually set $\delta \downarrow 0$.

Let $\eta\in C^\infty_0(I)$ be a real-valued function such that $\eta \, u =u$. We set 
\begin{equation}\label{lamb} 
\lambda(t) = \eta (t) \, \int_{t_0}^t \frac{\tilde{b}_\epsilon (\tau)}{p_\epsilon(\tau) + \delta}, 
\quad t \in I,
\end{equation} 
where $t_0\in I$, and $\delta>0$. Then clearly $\lambda \in  C^\infty_0(I)$, and 
\[
\lambda'(t) = \eta(t) \, \frac{\tilde{b}_\epsilon (t)}{p_\epsilon(t) + \delta} + 
\eta'(t) \, \lambda(t), \quad t \in I. 
\]

Notice that, on the support of $u$, we have  
\[
\tilde{b}_\epsilon (t) - \lambda'(t) \, (p_\epsilon(t) +\delta) (t) =0, \quad 
2 \, \tilde{b}_\epsilon (t) \,  \lambda'(t) - (p_\epsilon(t) +\delta) \, [\lambda'(t)]^2 = \frac{\tilde{b}_\epsilon (t)^2}{p_\epsilon(t) +\delta}. 
\quad 
\]

Using the exponential substitution $u = z e^{- i \lambda}$ discussed in 
Sec. \ref{grad} below, we deduce from \eqref{subs-ab} and \eqref{subs-b}
that  ${\rm Re} \, \langle -\mathcal{L_\epsilon} u, u\rangle \ge 0$ holds if and only if \eqref{2.1} holds with $p_\epsilon+ \delta$, 
$\tilde{b}_\epsilon$ and $\sigma_\epsilon$ in place of $p$,  $\tilde{b}$ and $\sigma$, 
for all small enough $\epsilon$ and $\delta>0$, that is 
\begin{equation}\label{2.1a}
\int_I (p_\epsilon(t)+ \delta) \,   h'(t)^2 \, dt -\langle \sigma_\epsilon,  h^2 
\rangle -  \int_{I}  \frac{\tilde{b}_\epsilon(t)^2}{p_\epsilon(t)+ \delta} \, h(t)^2 \, dt \ge 0,
\end{equation}
where $h \in C^{\infty}_0(I)$ has the same support as $u$, since $h$ is a linear combination of $f$ and $g$, the real and imaginary 
parts of $u$ (see Sec. \ref{proof-1}).

Clearly, $p_{\epsilon_k}  
\to p$  in $L^1_{{\rm loc}} (I)$, and $\sigma_{\epsilon_k}  \to \sigma$ in $D'(I)$ as $\epsilon_k\to 0$. 
Since $\tilde{b} \in L^1_{{\rm loc}} (I)$ and $p \in L^1_{{\rm loc}} (I)$, there exists 
a subsequence $k\to\infty$ so that  $\tilde{b}_{\epsilon_k}  
\to \tilde{b}$, and $p_{\epsilon_k}  
\to p$
a.e. Passing to the limit as $k\to\infty$, and using Fatou's lemma, we deduce 
the inequality 
\begin{equation}\label{2.1s}
\int_I (p(t)+ \delta) \,   h'(t)^2 \, dt -\langle \sigma,  h^2 
\rangle -  \int_{I}  \frac{\tilde{b}(t)^2}{p(t)+ \delta} \, h(t)^2 \, dt \ge 0. 
\end{equation}

Letting $\delta\downarrow 0$ and using the dominated convergence theorem and the 
monotone convergence theorem, we see  
that $\frac{\tilde{b}^2}{p} \in L^1_{{\rm loc}} (I)$, and \eqref{2.1} holds, provided 
$-\mathcal{L}$ is accretive. This proves the necessity of  condition \eqref{2.1} 
for the accretivity of the operator 
$-\mathcal{L}$.

To prove the sufficiency  of condition \eqref{2.1}, assuming $p\in L^1_{{\rm loc}}(I)$,  notice that 
it obviously yields \eqref{2.1s} for every $\delta>0$. Using the same mollification process as above we deduce 
\begin{equation}\label{2.1ss}
\int_I (p_\epsilon (t) + \delta) \,   h'(t)^2 \, dt -\langle \sigma_\epsilon,  h^2 
\rangle -  \int_{I}  \Big(\frac{\tilde{b}^2}{p + \delta}\Big)_\epsilon (t) \, h(t)^2 \, dt \ge 0. 
\end{equation}

By Jensen's inequality, we have 
\[
(\tilde{b}_\epsilon (t))^2 \le (p_\epsilon (t)+\delta) \Big(\frac{\tilde{b}^2}{p +\delta}\Big)_\epsilon (t), \quad t \in I. 
\]
Consequently, \eqref{2.1ss} yields \eqref{2.1a}. 
As was mentioned above,  inequality  \eqref{2.1a}, via the exponential substitution 
$u = z e^{- i \lambda}$ with $\lambda$ defined by \eqref{lamb}, is equivalent 
to the inequality ${\rm Re} \, \langle -\mathcal{L}_\epsilon u, u\rangle\ge 0$. 
Letting $\epsilon \to 0$, we conclude that $-\mathcal{L}$ is accretive. 
This proves statement (i) of Theorem \ref{Theorem 0}. 

To prove statement (ii), suppose that \eqref{0.2} holds for some $f \in L^1_{{\rm loc}}(I)$ such that $\frac{f^2}{p} \in L^2_{{\rm loc}}(I)$. Letting $g = \frac{f}{p}$, we see 
that $p \, g^2 \in L^1_{{\rm loc}}(I)$, and condition \eqref{g-cond} holds with 
$q$ in place of $\sigma$. Hence, by Theorem \ref{thm-real}, the operator 
$-\mathcal{N}$ is nonnegative definite, i.e., \eqref{0.1} holds, which yields that 
$-\mathcal{L}$ is accretive by statement (i) of Theorem \ref{Theorem 0}. 

 In the converse direction, if $-\mathcal{L}$ is accretive, then \eqref{0.1} holds by statement (i) of Theorem \ref{Theorem 0}. In other words, the operator 
  $-\mathcal{N}$ is nonnegative 
 definite. Thus, by Theorem \ref{thm-real}, there exists a function 
 $g \in L^2_{{\rm loc}}(I)$ such that \eqref{0.2} holds with $f= p \, g$. Here $p$ is uniformly bounded above and below by positive constants, so that $f \in L^2_{{\rm loc}}(I)$, and the right-hand side of \eqref{0.2} is well-defined. 
 The proof of  Theorem \ref{Theorem 0} 
is complete.\qed

\section{Decomposition of the drift term}\label{grad}

In this section, we deduce a version of Proposition \ref{prop_1} for 
vector fields 
\begin{equation}\label{E:4a}
 \tilde{\mathbf{b}} = P \nabla \lambda + \mathbf{d},
\end{equation}
where $\lambda$ and $\mathbf{d}$ are  real-valued. In particular, it 
yields more explicit criteria of accretivity in the special cases 
where $\mathbf{d}=0$, i.e., $ \tilde{\mathbf{b}} = P \nabla \lambda$, or 
$P=I$ and ${\rm div} \,  \mathbf{d}=0$, so that \eqref{E:4a}
 is the Hodge 
decomposition.

This is a consequence 
of the following theorem, which 
in a sense represents 
a higher dimensional analogue of Theorem \ref{Theorem 0} 
in the case  $n=1$, with $\lambda' =\frac{{\rm Im} \, b}{2p}$.  
\begin{theorem}\label{thm-grad} Let   $P$, $\tilde{\mathbf{b}}$, and $\sigma$ be defined by \eqref{expr}, where $P\in L^\infty_{{\rm loc}}(\Omega)^{n \times n}$  is
 a nonnegative definite matrix, and  
 $\tilde{\mathbf{b}}\in L^2_{{\rm loc}}(\Omega)^n$. Suppose 
 $\nabla \lambda \in L^{2}_{{\rm loc}}(\Omega)^n$, where 
 $\lambda \in D'(\Omega)$  is real-valued. Then  
$-\mathcal{L}$ is an accretive operator if and only if  the following two conditions  hold: 
 \begin{equation}\label{thm4-a}
 \begin{split}
 &[h]^2_{\mathcal{N}}  =  \int_{\Omega}  (P \nabla h \cdot \nabla h) \,  dx
  -\langle \sigma \, h, h \rangle
  \\& - \int_{\Omega} (2 \tilde{ \mathbf{b}} -P \nabla \lambda)\cdot \nabla \lambda  \, |h|^2 dx
\ge 0,
 \end{split}
\end{equation}
for all  $h \in C^{\infty}_0(\Omega)$, and 
 \begin{equation}\label{thm4-b}
 \left \vert \langle \tilde{\mathbf{b}} - P \nabla \lambda,  u \nabla v-  v \nabla  u \rangle
\right \vert \le [u]_{\mathcal{N}} \, [v]_{\mathcal{N}},
 \end{equation}
 for all real-valued $u, v \in C^{\infty}_0(\Omega)$.
\end{theorem}

\begin{remark} In the special case where $P$ is invertible (for instance, uniformly elliptic), and $P^{-1} \tilde{\mathbf{b}} = \nabla \lambda$ is a gradient field, the sole condition 
\eqref{thm4-a}, namely, 
\begin{equation}\label{thm4-as}
 [h]^2_{\mathcal{N}}  =  \int_{\Omega}  (P \nabla h \cdot \nabla h) \,  dx
  -\langle \sigma \, h, h \rangle-  \int_{\Omega} (P^{-1} \tilde{ \mathbf{b}} \cdot \tilde{ \mathbf{b}})  \, |h|^2 dx
\ge 0,
\end{equation}
for all  $h \in C^{\infty}_0(\Omega)$, characterizes accretive operators 
$-\mathcal{L}$. 

This is an analogue of condition \eqref{0.1} in the one-dimensional case. 
\end{remark}

\begin{remark} If $P=I$, then in decomposition \eqref{E:4a} 
we can pick the irrotational component 
of $\tilde{\mathbf{b}}$ as 
$\nabla \lambda$. 
In this case, Theorem \ref{thm-grad} is clearly equivalent to 
the inequality 
\begin{equation}\label{thm4-c}
 [h]^2_{\mathcal{N}}  =  || \nabla h||^2_{L^2(\Omega)} -\langle \sigma \,  h, h \rangle
   - \int_{\Omega}  \left( \, |\tilde{\mathbf{b}}|^2 - 
  |\mathbf{d}|^2 \, \right) \, 
  |h|^2 dx \ge 0, 
 \end{equation}
 for  all $h\in C^\infty_0(\Omega)$, where 
   $\mathbf{d}= \tilde{\mathbf{b}}-\nabla \lambda$ is  the divergence-free component 
   of $ \tilde{\mathbf{b}}$,  
 combined 
  with the   
   commutator inequality 
   \begin{equation}\label{thm4-d}
 \left \vert \langle \mathbf{d},  u \nabla v-  v \nabla  u \rangle
\right \vert \le [u]_{\mathcal{N}} \, [v]_{\mathcal{N}},
 \end{equation}
 for all real-valued $u, v \in C^{\infty}_0(\Omega)$.
 
Notice that here the condition $\tilde{\mathbf{b}}\in L^2_{{\rm loc}}(\Omega)$ is 
 necessary for $-\mathcal{L}$ to be accretive, provided $\mathbf{d}\in L^2_{{\rm loc}}(\Omega)$, as in the one-dimensional case where we can set $\mathbf{d}=0$. 
 \end{remark} 

\begin{proof} It follows from \eqref{L-2}, \eqref{L-2a} that, without loss of generality,  
 we may assume 
\[
\mathcal{L} u= {\rm div} \, (P \nabla u) + 2 i \, \tilde{\mathbf{b}} \cdot \nabla + \sigma,
\]
where 
$P$, $\tilde{\mathbf{b}}$, and $\sigma$ are given by \eqref{expr}.

Let us assume for simplicity that $P\in L^\infty_{{\rm loc}}(\Omega)^{n \times n}$, where $P$ is a nonnegative 
definite, symmetric,  real-valued $n \times n$ matrix, and 
$\nabla \lambda \in L^2_{{\rm loc}}(\Omega)$, where $ \lambda \in D'(\Omega)$
 is real-valued.

We use 
the substitution  $u=z \, e^{-i \lambda}$, where $u \in C^\infty_0(\Omega)$ is complex-valued,  to replace $\tilde{\mathbf{b}}$ with $\tilde{\mathbf{b}} - \nabla \lambda$. 

Suppose first that  $\lambda \in C^\infty_0(\Omega)$. Since $z =u \, e^{i \lambda} \in C^\infty_0(\Omega)$,
we have 
\[
P \nabla u = (P \nabla z - i z \, P \nabla \lambda) \, e^{- i \lambda}, 
\quad 
\nabla \bar u = (\nabla \bar z + i \bar z \, \nabla \lambda) \, e^{i \lambda}.
\]
Hence, 
\[
P \nabla u \cdot \nabla \bar u= P \nabla z \cdot \nabla \bar z + 
(P \nabla \lambda \cdot \nabla \lambda) \, |z|^2 
-i \, (P \nabla \lambda) \cdot   (z \, \nabla \bar z- \bar z \nabla z ).
\]
We deduce 
\begin{equation*}
\begin{split}
\int_{\Omega} (P\nabla u \cdot \nabla \bar u) \,  dx & = \int_{\Omega}  
(P\nabla z \cdot \nabla \bar z)  \, dx +  
\int_{\Omega} (P\nabla \lambda \cdot \nabla \lambda)  \, |z|^2 \, dx \\ & + 2  \, \int_{\Omega} P \nabla \lambda \cdot  {\rm Im} \, (z \nabla \bar z)  \, dx.
\end{split}
\end{equation*}

Since 
  $ \tilde{\mathbf{b}}  \in L^1_{{\rm loc}}(\Omega)$, we have 
\begin{equation*}
\begin{split}
 {\rm Re} \, \langle -\mathcal{L} u, u\rangle  & = \int_{\Omega}  \, 
 (P\nabla z \cdot \nabla \bar z) \, dx 
- \int_{\Omega} \Big [ \,  
2 (\tilde{\mathbf{b}} \cdot \nabla \lambda) -(P\nabla \lambda \cdot \nabla \lambda) \Big] \, |z|^2 \, dx \\ 
&  - \langle  \sigma, \, |z|^2   \rangle - 2 \langle  \tilde{\mathbf{b}}-  P\nabla \lambda,  {\rm Im} (z \, \nabla \bar z  )\rangle.
\end{split}
\end{equation*}
It follows that 
\[
{\rm Re} \, \langle -\mathcal{L} u, u\rangle \ge 0 \Longleftrightarrow  {\rm Re} \, \langle -\mathcal{M} z, z\rangle \ge 0, 
\]
where
\begin{equation*}
\mathcal{M} z = {\rm div} \, (P \nabla z)  + 2 i \, ( \tilde{\mathbf{b}} -  P \nabla \lambda) \cdot \nabla z + \Big(\sigma + 
2 (\tilde{\mathbf{b}} \cdot \nabla \lambda)-   (P\nabla \lambda \cdot \nabla \lambda)  \Big).
\end{equation*}

Thus, ${\rm Re} \, \langle -\mathcal{L} u, u\rangle \ge 0$ 
if and only if 
\begin{equation}\label{subs-ab}
\begin{split}
[h]_{\mathcal{M}}^2 = & \int_{\Omega} (P \nabla h \cdot \nabla h) \,  dx -\langle \sigma \,  h, h \rangle
  \\& -  \int_{\Omega} \Big [ \,  2 (\tilde{\mathbf{b}} \cdot \nabla \lambda) -(P\nabla \lambda \cdot \nabla \lambda)     \Big ] \, |h|^2 \, dx \ge 0,
 \end{split}
\end{equation}
for all  real-valued $h \in C^{\infty}_0(\Omega)$, and 
\begin{equation}\label{subs-b}
 \left \vert \langle \tilde{\mathbf{b}} - P \nabla \lambda,  u \nabla v-  v \nabla  u \rangle
\right \vert \le [u]_{\mathcal{M}} \, [v]_{\mathcal{M}},
 \end{equation}
 for all real-valued $u, v \in C^{\infty}_0(\Omega)$.

In the case $\nabla \lambda  \in L^2_{{\rm loc}}(\Omega)$, we notice that, without loss of generality 
we may assume that $\lambda$ is compactly supported in $\Omega$. Otherwise,  
we consider $\lambda \eta$, where $\eta \in C^{\infty}_0(\Omega)$ is a cut-off 
function such that $\eta \, u=u$, and apply the subsequent estimates to $\lambda \eta$.  We next replace 
$\lambda$ with its mollification $\lambda_\epsilon = \lambda \star \phi_\epsilon$, 
for $\epsilon>0$, 
where as usual $\phi_\epsilon(x)=\epsilon^{-n} \phi(\epsilon^{-1} x)$, for some 
$\phi \in C^\infty_0(\Omega)$. 

Using the same substitution as above, 
for $z \in C^{\infty}_0(\Omega)$, 
we set  
\[
u_{\epsilon}=z \, e^{-i \lambda_\epsilon} \in C^{\infty}_0(\Omega), 
\quad \nabla u_{\epsilon}= ( \nabla z- i \nabla \lambda_{\epsilon}) \, 
e^{- i \lambda_\epsilon}. 
\]
Notice that, as above,
\begin{equation*}
\begin{split}
  {\rm Re} \, \langle \mathcal{L} u_{\epsilon}, u_{\epsilon}\rangle  & = 
  \int_{\Omega}  \, (P \nabla z \cdot \nabla z) \, dx 
  -  \int_{\Omega} 
  \Big [ \,  2 (\tilde{\mathbf{b}} \cdot \nabla \lambda_{\epsilon})  - (P\nabla \lambda_{\epsilon} \cdot \nabla \lambda_{\epsilon}) 
 \Big ] \, |z|^2  \, dx 
 \\ &        - \langle \sigma, |z|^2\rangle  -  
2 \langle  \tilde{\mathbf{b}} - P \nabla \lambda_{\epsilon}, {\rm Im} (z \, \nabla \bar z  )\rangle \ge 0,
\end{split}
\end{equation*}
which yields the following two conditions:
\begin{equation*}
\begin{split}
[z]^2_{\mathcal{M}} & =  \int_{\Omega}  \, (P \nabla z \cdot \nabla z) \, dx 
  -  \int_{\Omega} 
  \Big [ \,  2 (\tilde{\mathbf{b}} \cdot \nabla \lambda_{\epsilon})  - (P\nabla \lambda_{\epsilon} \cdot \nabla \lambda_{\epsilon}) 
 \Big ] \, |z|^2  \, dx     \\ &  - \langle \sigma, |z|^2\rangle  \ge 0, \quad \left \vert \langle  \tilde{\mathbf{b}} - P \nabla \lambda_{\epsilon}, u \nabla v -v \nabla u 
\right \vert   \le [u]_{\mathcal{M}} [v]_{\mathcal{M}}, 
\end{split}
\end{equation*}
for all  $z\in C^\infty_0(\Omega)$ (real- or complex-valued) and real-valued $u, v \in C^\infty_0(\Omega)$. 

We have $||(\nabla \lambda_{\epsilon} - \nabla \lambda) \eta ||_{L^2(\Omega)} \to 0$ 
as  $\epsilon \to 0$. Consequently, it follows $||( \tilde{\mathbf{b}} \cdot \nabla \lambda_{\epsilon}-  \tilde{\mathbf{b}} \cdot \nabla \lambda  )\eta ||_{L^2(\Omega)} \to 0$ for $\tilde{\mathbf{b}}   \in L^2_{{\rm loc}}(\Omega)$. 
Since $P\in L^\infty_{{\rm loc}}(\Omega)^{n \times n}$, we have  
$||(P \nabla \lambda_{\epsilon} - P \nabla \lambda) \eta ||_{L^2(\Omega)} \to 0$ as well. 
Passing to the limit as $\epsilon \to 0$ completes the proof of Theorem 
\ref{thm-grad}.
\end{proof}

\section{{\rm BMO} estimates, trace inequalities,  and admissible measures}\label{Section 2}

In this section, we discuss  {\rm BMO} estimates, trace inequalities and admissible measures  
used in Theorem \ref{main}, which gives necessary and sufficient conditions 
on  $A$, 
$\mathbf{b} $  
and $c$ for the accretivity of   $-\mathcal{L}$ on $\R^n$,  under 
some additional assumptions on the upper and lower bounds of the quadratic 
forms $[\cdot]_{\mathcal{H}}$ imposed in Sec. \ref{sec2.5}.

By $L^{1, \, 2}(\Omega)$ we denote the energy space (homogeneous Sobolev space)
 defined  as the completion of the complex-valued   
$C^\infty_0(\Omega)$ functions in the Dirichlet norm 
$||\nabla \cdot ||_{L^2(\Omega)}$.

For $f \in L^1_{\rm loc} (\R^n)$, we set 
$
m_B (f) = \frac 1 {|B|} \int_B f(x) \, dx,
$
where $B$ is a ball in $\R^n$, and denote by  
 ${\rm BMO}(\R^n)$ the class of functions 
$f \in L^r_{\rm loc} (\R^n)$ for which 
$$
\sup_{x_0\in \R^n, \, \delta>0} \, \,  \frac {1}{|B_\delta(x_0)|} 
\int_{B_\delta (x_0)} |f(x)-m_{B_\delta(x_0)}(f)|^r \, dx < + \infty,
$$
for any (or, equivalently, all) $1\le r < +\infty$.

The corresponding vector- and matrix-valued function spaces are introduced in 
a similar way. In particular,  ${\rm BMO}(\R^n)^n$ stands for 
 the class of vector fields 
$\mathbf{f}= \{f_j\}_{j=1}^n : \, \R^n \to \C^n$, such that 
$f_j \in {\rm BMO}(\R^n)$, $j =1, 2, \dots, n$. The 
 matrix-valued analogue is denoted by ${\rm BMO}(\R^n)^{n \times n}$, etc.  
 The notion of the weak-$*$ ${\rm BMO}$-convergence is discussed based on the 
 $H^1-{\rm BMO}$ duality in 
\cite[Ch. IV]{St}.

The matrix divergence 
operator  
${\rm Div}: \, D'(\Omega)^{n \times n}\to D'(\Omega)^{n}$ is defined 
on matrix fields $F = (f_{i j})_{i, j =1}^n \in D'(\Omega)^{n \times n}$ by ${\rm Div} \, F = \left ( \sum_{j=1}^n \, \partial_j \, 
f_{i j}  \right)_{i=1}^n \in D'(\Omega)^n$. If $F$ is 
skew-symmetric, i.e., $f_{ij} = -f_{ji}$, 
then we obviously have ${\rm div} \, ({\rm Div} \, F) = 0$.

The Jacobian, $\D: \, D'(\Omega)^{n} \to D'(\Omega)^{n \times n}$, is 
the formal adjoint of $-{\rm Div}$, 
$$\langle  {\rm Div} \, F,  \, \mathbf{v}  \rangle = -  \, 
\langle F, \, \D \, \mathbf{v}  
\rangle, \qquad   \, \mathbf{v}   \in C^\infty_0(\Omega)^n.$$
Here the scalar product of matrix fields $F= (f_{i j})_{i, j =1}^n$ and 
$G= (g_{i j})_{i, j =1}^n$ is defined by 
$\langle F, \, G\rangle = \sum_{i, j =1}^n \langle f_{ij}, \, g_{ij}\rangle.$ If $F, \, G \in L^2(\Omega)^{n\times n}$, 
then 
$$\langle F, \, G\rangle = \int_{\Omega} {\rm trace} \, (F^t \cdot \bar G) \, dx,$$
where $F^t = (f_{ji})_{i, j =1}^n$ is the transposed matrix, and $\bar G = (\bar g_{i j})_{i, j =1}^n$.

The matrix curl operator ${\rm Curl}: \, D'(\Omega)^n \to D'(\Omega)^{n \times n}$ 
is defined on vector fields $\mathbf{f}=(f_k)_{k=1}^n$ 
 by ${\rm Curl} \, \mathbf{f} = (\partial_j f_k-\partial_j f_k)_{j, k=1}^n$. Clearly, 
 ${\rm Curl} \, \mathbf{f}$ is always a skew-symmetric matrix field. 

Notice that in the case $n=3$ we can use the usual vector-valued ${\rm curl}$ operator which maps $D'(\Omega)^3\to D'(\Omega)^3$. For instance, if a vector field  is 
represented as $\mathbf{b}= {\rm curl} (\mathbf{g})$, then we can write 
  commutator inequalities of the type  
$
\left \vert \langle \mathbf{b}, u \nabla \, v -v \, \nabla u\rangle \right \vert \le C \, 
|| \nabla u||_{L^2(\R^3)} \, || \nabla v ||_{L^2(\R^3)}, 
$
in the equivalent form 
\[
\left \vert \langle \mathbf{g}, \nabla u \times \nabla v\rangle \right \vert \le C \, 
|| \nabla u||_{L^2(\R^3)} \, || \nabla v ||_{L^2(\R^3)}, 
\]
 where $\mathbf{g}  \in {\rm BMO}(\R^3)^3$. This is  an analogue of the Jacobian determinant inequality \eqref{jacob} in two dimensions. 
Such inequalities are studied in compensated compactness theory \cite{CLMS}. 

Similarly, when $n=3$, in Theorem \ref{thm-i2}   
we can use the Hodge decomposition in $\R^3$, 
 \[
 \tilde{\mathbf{b}} =\nabla f + {\rm curl} (\mathbf{g}), 
 \]
 where 
 $\mathbf{g} =\Delta^{-1} ( {\rm curl} \, \tilde{\mathbf{b}}) \in {\rm BMO}(\R^3)^3$. Here  
 the operator $\Delta^{-1}$ is understood in the sense of the weak-$*$ ${\rm BMO}$-convergence, as explained in \cite{MV3}. Notice that this decomposition does not 
 contain  any harmonic vector fields $\mathbf{h}$ such that both ${\rm div} \, (\mathbf{h})=0$  and ${\rm curl} (\mathbf{h})=0$. This is, of course, true in $\R^n$ for higher 
 dimensions $n\ge 4$ as well.

The capacity of a compact set $e\subset \R^n$ is defined by 
(\cite{M}, Sec. 2.2): 
\begin{equation}\label{E:cap}
{\rm cap} \, (e) = \inf \, 
\left\{ \, ||u||^2_{L^{1,2} (\R^n)}  : \quad 
u \in C^\infty_0(\R^n), \quad u(x) \ge 1 \, \,  {\rm on} 
\, \,  e \right\}.
\end{equation}
For a cube or ball $Q$ in $\R^n$, 
\begin{equation}\label{E:cube}
{\rm cap} \, (Q) \simeq |Q|^{1-\frac 2 n} \quad {\rm if} \, \,  n \ge 3; 
\quad {\rm cap} \, (Q) =0 \quad  {\rm if} \, \, n=2.
\end{equation}

The capacity ${\rm Cap} \, (\cdot)$ associated with the inhomogeneous 
Sobolev space $W^{1, \, 2}(\R^n)$ defined  by
\begin{equation}\label{E:Cap}
{\rm Cap} \, (e) = \inf \, 
\left\{ \, ||u||^2_{W^{1,2} (\R^n)}  : \quad 
u \in C^\infty_0(\R^n), \quad u(x) \ge 1 \, \,  {\rm on} 
\, \,  e \right\},
\end{equation}
for compact sets $e\subset \R^n$. 
Note that ${\rm Cap} \, (e) \simeq {\rm cap} \, (e)$ if ${\rm diam} \, (e) \le 1$, and $n \ge 3$.  
For a cube or ball $Q$ in $\R^n$,  
\begin{equation}\label{E:Cube}
{\rm Cap} \, (Q) \simeq |Q|^{1-\frac 2 n} \quad {\rm if} \, \,  n \ge 3; 
\quad {\rm Cap} \, (Q) \simeq \left (\log \tfrac 2 {|Q|}\right)^{-1} \quad  {\rm if} \, \, n=2,
\end{equation}
provided $|Q| \le 1$. For these and other properties of capacities, as well as related notions of 
potential theory we refer to \cite{AH},  \cite{M}. 

By $\mathcal{M}^{+}(\Omega)$ we denote the calls of all nonnegative Radon measures (locally finite) in an open set $\Omega \subseteq \R^n$. We discuss in this section several equivalent characterizations of the class of \textit{admissible 
measures} $\mu \in \mathfrak{M}_+^{1, \,2}(\Omega)$  which obey the so-called trace inequality \eqref{E:tr-om} (see \cite{AH}, \cite{M}, and the extensive literature cited there). 

We start with the case $\Omega=\R^n$. 
A measure $\mu \in \mathcal{M}^{+}(\R^n)$ is said to be admissible, i.e., 
$\mu \in \mathfrak{M}_+^{1, \,2}(\R^n)$, if it  
obeys the trace inequality: 
\begin{equation}\label{E:tr0}
 \left(\int_{\R^n} |u|^2 \,  d \mu\right)^{\frac{1}{2}} \le C \, || \nabla u||_{L^2(\R^n)}, \qquad u \in C^\infty_0(\R^n),
\end{equation}
where $C$ is a positive constant which does not depend on $u$.

For $\mu\in \mathcal{M}^{+}(\R^n)$, we denote
by $I_1 \mu=(-\Delta)^{-\frac 1 2} \mu$  the Riesz potential of order $1$, 
\[
I_1 \mu(x)=(-\Delta)^{-\frac 1 2} \mu(x) = c(n) \, \int_{\R^n} \, \frac{d \mu(y)}
{|x-y|^{n-1}}, \quad x \in \R^n.
\]
 Here $c(n)$ is a normalization constant which depends only on $n$.    

We have the following equivalent characterizations of  admissible measures in $\R^n$ (see \cite[Ch. 11]{M}).

\begin{theorem}\label{Trace Theorem} Let $\mu\in \mathcal{M}^{+}(\R^n)$. Then  $\mu \in  \mathfrak{M}_+^{1, \,2}(\R^n)$ if and only if any one of the following 
statements holds.

{\rm (i)}  The Riesz potential $I_1 \mu \in L^2_{{\rm loc}} (\R^n)$, and 
$(I_1 \mu)^2 \in \mathfrak{M}_+^{1, \,2}(\R^n)$, i.e., 
 \begin{equation}\label{E:tr1}
 \left(\int_{\R^n} |u|^2 \, (I_1 \mu)^2 \,  d x \right)^{\frac{1}{2}}\le c_1 \, 
 || \nabla u||_{L^2(\R^n)}, 
 \quad u \in C^\infty_0(\R^n),
\end{equation}
where $c_1>0$ does not depend on $u$.

{\rm (ii)} For every compact set $e \subset \R^n$, 
 \begin{equation}\label{E:tr2}
 \mu (e) \le c_2 \, \text{\rm{cap}} \, (e),
\end{equation}
where $c_2$ does not depend on $e$.

{\rm (iii)} For every  ball $B$ in $\R^n$,
 \begin{equation}\label{E:tr3}
 \int_B (I_1 \mu_B)^2 \, dx \le c_3 \, \mu (B),
\end{equation}
where $d \mu_B = \chi_B \, d \mu$, 
and $c_3$ does not depend on $B$.

{\rm (iv)} The pointwise inequality
 \begin{equation}\label{E:tr4}
I_1 [(I_1 \mu)^2 (x)] \le c_4 \, I_1 \mu(x) < \infty 
\end{equation}
holds a.e., where $c_4$ does not depend on $x \in \R^n.$ 

{\rm (v)} For every dyadic cube $P$ in  $\R^n$, 
\begin{equation}\label{E:tr5}
\sum_{Q \subseteq P}   \frac {\mu(Q)^2}{
|Q|^{1 - \frac 2 n}}  \le c_5 \, \mu (P),
\end{equation}
where the sum is taken over all dyadic cubes $Q$ contained in $P$, and 
$c_5$ does not depend on $P$.

Moreover, the least constants $c_i$, $i=1, \ldots, 5$,  are equivalent to the least constant $c$ in {\rm (\ref{E:tr0})}. 
\end{theorem} 

 \begin{remark}  It follows from Poincar\'{e}'s inequality 
 and the formula for the capacity of a ball, ${\rm cap} \, (B(x, r))=c_n \, r^{n-2}$,  
 for $n \ge 3$,  
 that if $d \mu =|\nabla v|^2 dx \in \mathfrak{M}_+^{1, \,2}(\R^n)$, 
 where $v \in L^{1,2}_{{\rm loc}}(\R^n)$, then $v \in {\rm BMO} (\R^n)$. 
 \end{remark}

 \begin{remark} An analogous characterization holds for admissible measures on the Sobolev space $W^{1, \, 2}(\R^n)$ 
in place of $L^{1, \, 2}(\R^n)$. One only needs to replace Riesz potentials $I_1\mu=(-\Delta)^{-\frac 1 2} \mu$ in 
statements {\rm (i)},  {\rm (iii)},  and {\rm (iv)}  
by Bessel potentials 
 $J_1\mu=(1-\Delta)^{-\frac 1 2} \mu$, the capacity ${\rm cap} \, (\cdot)$ in   {\rm (ii)}
   by ${\rm Cap} \, (\cdot)$, and restrict oneself to cubes $P$ such that $|P| \le 1$ in 
{\rm (v)}. 
\end{remark}  

Originally, the trace inequality in the Sobolev space $L^{1,2}_0(\Omega)$,  
for an arbitrary open set $\Omega \subseteq \R^n$, was characterized by 
the first author 
in \cite{M1}, \cite{M2} in capacity terms as follows. A  measure 
$\mu \in \mathcal{M}^{+}(\Omega)$ is said to be admissible if the inequality
\begin{equation}\label{E:tr-om}
 \left (\int_{\Omega} |u|^2 \,  d \mu\right)^{\frac{1}{2}} \le C \, || \nabla u||_{L^2(\Omega)} 
\end{equation}
holds for all $u \in C^\infty_0(\Omega)$, where $C$ is a positive constant which does not depend on $u$. The class of 
admissible measures for \eqref{E:tr-om} is denoted by $\mathfrak{M}^{1,2}(\Omega)$. 

The capacity ${\rm cap} \, (e, \Omega)$ of a compact subset $e \subset \Omega$ 
is defined by (see \cite{M}, Sec. 2.2): 
\begin{equation}\label{E:cap-om}
{\rm cap} \, (e, \Omega) = \inf \, 
\left\{ \, ||\nabla u||^2_{L^2 (\Omega)}  : \quad 
u \in C^\infty_0(\Omega), \quad u(x) \ge 1 \, \,  {\rm on} 
\, \,  e \right\}.
\end{equation}

Then $\mu \in \mathfrak{M}^{1,2}(\Omega)$ if and only if (\cite{M1}, \cite{M2}; see also \cite[Sec. 2.3]{M}) 
\begin{equation}\label{cap-crit}
\mu(e) \le c \, {\rm cap} \, (e, \Omega), 
\end{equation}
where the constant $c$ does not depend on $e$. 

Moreover, condition \eqref{cap-crit} with $c=\frac{1}{4}$  is sufficient for 
\eqref{E:tr-om} to hold with $C=1$. Conversely, 
 \eqref{cap-crit} with $c=1$ is necessary in order that \eqref{E:tr-om} hold with $C=1$. 
Both constants $c=\frac{1}{4}$ and $c=1$ and in these statements are sharp 
(see \cite[Sec. 2.5.2]{M}).

There is a dual characterization of the trace inequality which does not use capacities. 
Let us assume that $G$ is a nontrivial nonnegative Green's function associated with the Dirichlet Laplacian 
in $\Omega$. Then $\mu \in \mathfrak{M}^{1,2}(\Omega)$ if and only if the inequality 
\begin{equation}\label{energy-crit}
\int_{e \times e} G(x, y) \, d \mu(x) \, d \mu(y) \le c \, \mu(e)
\end{equation}
holds for all measurable sets $e\subset \Omega$. 

Moreover, inequality \eqref{energy-crit} is equivalent to the weighted norm inequality 
\[
|| G (f d \mu)||_{L^2(\Omega, \mu)} \le C \, || f||_{L^2(\Omega, \mu)}, \quad \textrm{for all} \, \, f \in L^2(\Omega, \mu). 
\] 
It is also equivalent to the weak-type $(1, 1)$ inequality 
\[
|| G (f d \mu)||_{L^{1, \infty}(\Omega, \mu)} \le C \, || f||_{L^1(\Omega, \mu)}, \quad \textrm{for all} \, \, f \in L^1(\Omega, \mu). 
\] 
Here $G (f d \mu)(x) =\int_{\Omega} G(x, y) \, f(y) \, d \mu(y)$ is Green's potential 
of $f \, d \mu$.  

In \cite[Theorem 6.5]{QV}, 
similar  results are proved for nonnegative kernels $G$ satisfying a weak form of the 
\textit{maximum principle}: 
\[
\sup \{ G \nu(x): \, \, x \in \Omega\} \le \mathfrak{b} \, \sup \{ G \nu(x): \, \, x \in \textrm{supp} \, \nu\},
\]
where $ \mathfrak{b}\ge 1$ is a constant which does not depend on 
$\nu \in \mathcal{M}^{+}(\Omega)$. 

In particular, if  $G$ is a \textit{quasi-metric} kernel, i.e., $d(x, y)=\frac{1}{G(x, y)}$ 
is symmetric and satisfies a quasi-triangle inequality, then it suffices to verify  \eqref{energy-crit}  on quasi-metric balls 
 $B(x, r)=\{y \in \Omega: \, \, d(x, y)\le r\}$  in place of arbitrary sets $e$ (see \cite{FNV}, \cite{QV}). 
 
Analogous results hold (\cite{FNV}) for a more general class of quasi-metrically modifiable kernels $G$. This is important since the Green kernel $G$ is known to be quasi-metrically modifiable if $\Omega$ satisfies the boundary Harnack principle,  for instance,  if $\Omega$ is a bounded NTA domain (\cite{K}).

\end{document}